\documentclass{amsart}

\issueinfo{61}{1}{January}{2024}
\dateposted{October 6, 2023}

\pagespan{3}{22}
\copyrightinfo{2023}{by Henry Cohn}
\PII{\url{https://doi.org/10.1090/bull/1813}}

\usepackage[pdftex,colorlinks,citecolor=black,linkcolor=black,urlcolor=black,bookmarks=false]{hyperref}

\usepackage[msc-links,nobysame]{amsrefs}

\usepackage{tikz}
\usepackage{pgfplots}
\pgfplotsset{compat=1.18}

\newcommand{\arXiv}[1]{\texttt{arXiv:\href{https://arXiv.org/abs/#1}{#1}}}

\newcommand{\R}{{\mathbb R}}

\newcommand{\Z}{{\mathbb Z}}
\newcommand{\C}{{\mathbb C}}
\newcommand{\Hyp}{{\mathbb H}}
\newcommand{\vol}{\mathop{\textup{vol}}}
\renewcommand{\Im}{\mathop{\textup{Im}}}

\newcommand{\supp}{\mathop{\textup{supp}}}

\newtheorem{theorem}{Theorem}[section]
\newtheorem{lemma}[theorem]{Lemma}
\newtheorem{proposition}[theorem]{Proposition}

\theoremstyle{definition}

\theoremstyle{remark}

\numberwithin{equation}{section}
\numberwithin{figure}{section}

\begin{document}

\title{From sphere packing to Fourier interpolation}

\author{Henry Cohn}
\address{Microsoft Research New England, One Memorial Drive, Cambridge, Massachusetts
02142}
\email{cohn@microsoft.com}


\date{July 17, 2023}

\begin{abstract}
Viazovska's solution of the sphere packing problem in eight dimensions is
based on a remarkable construction of certain special functions using
modular forms. Great mathematics has consequences far beyond the problems
that originally inspired it, and Viazovska's work is no exception. In this
article, we'll examine how it has led to new interpolation theorems in
Fourier analysis, specifically a theorem of Radchenko and Viazovska.
\end{abstract}

\maketitle

\section{Sphere packing}

The sphere packing problem asks how densely congruent spheres can be packed
in Euclidean space. In other words, what fraction of space can be filled with
congruent balls, if their interiors are required to be disjoint?\footnote{To
make this question precise, we could take the limit as $r \to \infty$ of
the density for packing unit spheres in a sphere of radius $r$, or a cube of
side length $r$. We would obtain the same limit for any reasonable container
(see, for example, \cite{BHS}).} Everyone can pack spheres intuitively in low
dimensions: the optimal two-dimensional packing is a hexagonal arrangement,
and optimal three-dimensional packings are stacks of optimal two-dimensional
layers, nestled together as closely as possible into the gaps in the layers
(see Figure~\ref{figure:packing}).

In fact, these packings are known to be optimal. The two-dimensional problem
was solved by Thue \cites{Thu1,Thu2}, with a more modern proof by Fejes T\'oth
\cite{F}, and the three-dimensional problem was solved by Hales \cite{H}. The
two-dimensional proof is not so complicated, but the three-dimensional proof
is difficult to check, because it relies on both enormous machine
calculations and lengthy human arguments in a sequence of papers. To give a definitive
demonstration of its correctness, Hales and a team of collaborators have
produced a formally verified proof \cite{Hplus}, i.e., a proof that has been
algorithmically verified using formal logic.

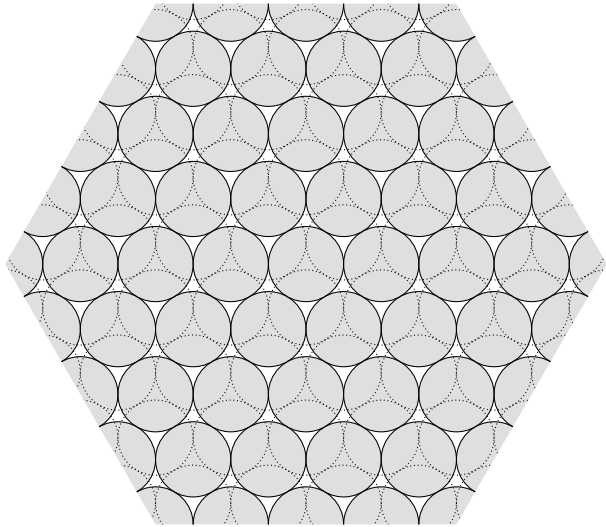
\begin{figure}
\centering
\begin{tikzpicture}[scale=0.5]
\clip (-8,0)--({-8+12/3},{12/3*sqrt(3)})--({8-12/3},{12/3*sqrt(3)})--(8,0)--({8-12/3},{-12/3*sqrt(3)})--({-8+12/3},{-12/3*sqrt(3)})--cycle;
\foreach \i in {-10,...,10}
\foreach \j in {-10,...,10}
{
\fill[white!87.5!black] ({2*\i+\j},{sqrt(3)*\j}) circle (1);
}
\foreach \i in {-10,...,10}
\foreach \j in {-10,...,10}
{
\draw  ({2*\i+\j},{sqrt(3)*\j}) circle (1);
\draw[densely dotted]  ({2*\i+\j+1},{sqrt(3)*\j+sqrt(3)/3}) circle (1);
}
\end{tikzpicture}
\caption{A two-dimensional cross section of an optimal three-dimensional sphere packing, with dotted lines indicating spheres in an adjacent layer.}
\label{figure:packing}
\end{figure}

On the one hand, the solution of the three-dimensional sphere packing problem
is a triumph of modern mathematics, a demonstration of humanity's ability to
overcome even tremendously challenging obstacles. On the other hand, to a
general audience it can sound like a parody of pure mathematics, in which
mathematicians devote immense efforts to proving an intuitively obvious
assertion. It's natural to feel discouraged about the future of a subfield in
which it's easy to guess the answer and almost impossible to prove it. For
comparison, a rigorous solution of the four-dimensional sphere packing
problem remains far out of reach. If the difficulty increases as much from
three to four dimensions as it did from two to three, then humanity may never
see a proof.

One noteworthy change as we move to higher dimensions is that we lose much of
our intuition, and the answer is no longer easy to guess. For example, it is not always true that we can obtain an
optimal packing in $\R^n$ by stacking optimal $(n-1)$-dimensional layers (see
\cite{CS} for details). In sufficiently high dimensions, there are no
conjectures for optimal packings, the best upper and lower bounds known for
the packing density differ by an exponential factor in the dimension, and we
cannot even predict whether the densest packings should be crystalline or
disordered. In short, we know shockingly little about how spherical particles
behave in high dimensions. Of course this means there are plenty of
intriguing phenomena to explore.

Certain dimensions stand out in the midst of this ignorance as having
exceptionally dense packings. The most amazing of all are eight and
twenty-four dimensions, which feature the $E_8$ root lattice and the Leech
lattice $\Lambda_{24}$. (We will not construct these lattices here; see
\cites{E, Tho, SPLAG} for constructions.) Recall that a lattice in $\R^n$ is
just a discrete subgroup of rank $n$; in other words, for each basis
$v_1$, \dots, $v_n$ of $\R^n$, the set
\[
\{ a_1 v_1 + \dots + a_n v_n : a_1,\dots,a_n \in \Z\}
\]
is a lattice. Every lattice leads to a sphere packing by centering congruent
spheres at the lattice points, with the radius chosen as large as possible
without overlap. Lattice packings are common in low dimensions, but
there is no reason to expect an optimal packing to have this sort of algebraic
structure in general. For example, in $\R^{10}$ the best packing
known, the aptly named Best packing \cite{B}, has density more than 8\%
greater than any known lattice packing in $\R^{10}$. By contrast, the $E_8$
and Leech lattices yield
impressively dense packings with extraordinary symmetry groups, and their
density and symmetry are so far out of the ordinary that it is difficult to
imagine how they could be improved.

In 2016 Maryna Viazovska \cite{V} solved the sphere packing problem in $\R^8$
with an innovative use of modular forms, which was soon extended to $\R^{24}$
as well~\cite{CKMRV1}; both $E_8$ and the Leech lattice do indeed turn out to
be optimal sphere packings. These are the only cases in which the sphere
packing problem has been solved above three dimensions. Although the proofs
require more machinery than those in two or three dimensions, most notably
the theory of modular forms, they are much shorter and simpler than one might
fear. Viazovska's proof dispelled the gloomy possibility that
higher-dimensional sphere packing could be beyond human understanding, and
she was awarded a Fields Medal in 2022 for this line of work.

In addition to her breakthrough in sphere packing, Viazovska's modular form
techniques have led to unexpected consequences, such as interpolation
theorems showing that a radial function $f$ can be reconstructed from the
values of $f$ and its Fourier transform $\widehat{f}$ on certain discrete
sets of points. Although Fourier interpolation may sound rather far afield
from sphere packing, it turns out to be closely connected. In this article,
we'll explore how Viazovska's work led to this connection and how to prove
a fundamental interpolation theorem of Radchenko and
Viazovska \cite{RV}. For comparison, \cite{dLV}, \cite{C1}, \cite{V2},
\cite{V3}, and \cite{C2} are expositions of her work that focus on other
themes.

$\phantom{}$

\section{From sphere packing to Fourier analysis}

The connection between packing problems and Fourier analysis originated in the
work of Delsarte \cite{D} on linear programming bounds for error-correcting
codes. For sphere packings in Euclidean space, a continuous analogue of
Delsarte's work was developed by Cohn and Elkies \cite{CE}. The quality of
this bound depends on the choice of an auxiliary function satisfying certain
inequalities, and Viazovska's breakthrough amounted to figuring out how to
optimize that choice.

We will normalize the Fourier transform of an integrable function $f \colon
\R^n \to \C$ by
\[
\widehat{f}(y) = \int_{\R^n} f(x) e^{-2\pi i \langle x,y \rangle}\, dx,
\]
where $\langle \cdot, \cdot \rangle$ denotes the usual inner product on
$\R^n$. We'll generally restrict our attention to Schwartz functions, i.e.,
infinitely differentiable functions $f$ such that for all real numbers $c>0$
and nonnegative integers $i_1,\dots,i_n$,
\[
\left| \frac{\partial^{i_1+\dots+i_n}}{\partial x_1^{i_1} \cdots\partial x_n^{i_n} } f(x_1,\dots,x_n) \right| = O(|x|^{-c})
\]
as $|x| \to \infty$. These smoothness and decay conditions can be somewhat
weakened in each application below, but Schwartz functions are the
best-behaved case. We'll also frequently study radial functions, i.e.,
functions $f$ for which $f(x)$ depends only on $|x|$, in which case we will
write $f(r)$ for $r \in [0,\infty)$ to denote the value $f(x)$ with $|x|=r$
and $f'$ for the radial derivative of $f$. Note that the spaces of radial
functions and of Schwartz functions are both preserved by the Fourier
transform.

The linear programming bound is the following method for producing a density
bound from a suitable auxiliary function $f$. The name ``linear programming
bound'' refers to the fact that optimizing this bound can be recast as an
infinite-dimensional linear programming problem (i.e., linear optimization
problem).

\begin{theorem}[Cohn and Elkies \cite{CE}] \label{thm:lpbound}
Let $f \colon \R^n \to \R$ be a radial Schwartz function and $r$ a positive
real number such that
\begin{enumerate}
\item $f(x) \le 0$ whenever $|x| \ge r$,

\item $\widehat{f}(y) \ge 0$ for all $y$, and

\item $f(0) = \widehat{f}(0)=1$.
\end{enumerate}
Then the optimal sphere packing density in $\R^n$ is at most the volume
$\vol(B_{r/2}^n)$ of a ball of radius $r/2$ in $\R^n$.
\end{theorem}

It is far from obvious how to produce good auxiliary functions $f$ for use in
this theorem, or how to optimize the choice of $f$, i.e., minimize $r$. In
fact, the exact optimum is known only for $n=1$, $8$, and $24$. However, one
can perform a numerical optimization over a suitable space of functions, such
as polynomials of fixed degree times a Gaussian, with the hope that it will
converge to the global optimum as the degree tends to infinity.
Figure~\ref{figure:plot} compares the resulting numerical bound with the
density of the best packing known.

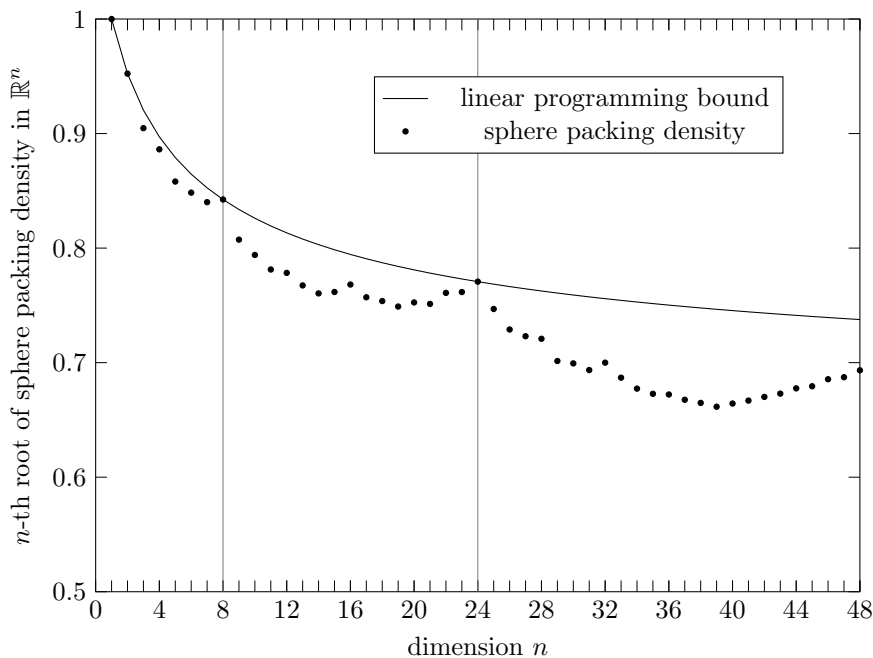
\begin{figure}
\centering
\begin{tikzpicture}
\begin{axis}[every x tick/.style={black}, every y tick/.style={black},
scale only axis,
xmin=0, xmax=48, ymin=0.5, ymax=1,
xlabel=\textup{dimension $n$},
ylabel=\textup{$n$-th root of sphere packing density in $\mathbb{R}^n$},
domain = 0:48,
enlargelimits = false,
xtick = {0,4,...,48},
extra x ticks = {1,...,48},
extra x tick labels={},
legend style={at={(0.9,0.9)},anchor=north east},
width=4in,
height=3in
]
\draw[black!50!white] (axis cs:8,0.5) -- (axis cs:8,1);
\draw[black!50!white] (axis cs:24,0.5) -- (axis cs:24,1);
\addplot +[mark=none, color=black] plot coordinates {
(1,1.0000000)
(2,0.95231281)
(3,0.92041678)
(4,0.89710713)
(5,0.87908032)
(6,0.86458190)
(7,0.85258181)
(8,0.84242944)
(9,0.83369039)
(10,0.82606176)
(11,0.81932483)
(12,0.81331708)
(13,0.80791487)
(14,0.80302217)
(15,0.79856310)
(16,0.79447678)
(17,0.79071375)
(18,0.78723327)
(19,0.78400151)
(20,0.78099005)
(21,0.77817486)
(22,0.77553542)
(23,0.77305413)
(24,0.77071574)
(25,0.76850702)
(26,0.76641637)
(27,0.76443363)
(28,0.76254980)
(29,0.76075692)
(30,0.75904788)
(31,0.75741635)
(32,0.75585664)
(33,0.75436362)
(34,0.75293269)
(35,0.75155964)
(36,0.75024070)
(37,0.74897239)
(38,0.74775156)
(39,0.74657532)
(40,0.74544104)
(41,0.74434628)
(42,0.74328880)
(43,0.74226653)
(44,0.74127758)
(45,0.74032017)
(46,0.73939267)
(47,0.73849356)
(48,0.73762142)
};
\addplot +[mark=*, only marks, mark options={fill=black,scale=0.5},color=black] plot
coordinates {
(1,1.0000000)
(2,0.95231281)
(3,0.90469990)
(4,0.88622693)
(5,0.85810157)
(6,0.84841426)
(7,0.84008437)
(8,0.84242944)
(9,0.80737603)
(10,0.79402251)
(11,0.78131821)
(12,0.77836551)
(13,0.76740823)
(14,0.76044358)
(15,0.76170417)
(16,0.76819664)
(17,0.75704214)
(18,0.75375437)
(19,0.74898436)
(20,0.75255685)
(21,0.75129107)
(22,0.76091135)
(23,0.76161908)
(24,0.77071574)
(25,0.74684243)
(26,0.72894781)
(27,0.72304974)
(28,0.72085193)
(29,0.70141757)
(30,0.69934523)
(31,0.69354272)
(32,0.69999301)
(33,0.68688018)
(34,0.67733090)
(35,0.67279863)
(36,0.67218873)
(37,0.66759512)
(38,0.66486383)
(39,0.66154449)
(40,0.66430723)
(41,0.66687565)
(42,0.67008024)
(43,0.67296813)
(44,0.67754247)
(45,0.67939385)
(46,0.68550577)
(47,0.68727962)
(48,0.69332873)
};
\legend{\quad linear programming bound, \quad sphere packing density}
\end{axis}
\end{tikzpicture}
\caption{A plot of the numerically computed linear programming bound \cite{ACHLT}
and the best sphere packing density currently known \cite{SPLAG}. The plot shows
the $n$-th root of the density in dimension $n$, with $n=8$ and $n=24$ marked by
vertical lines.}
\label{figure:plot}
\end{figure}

In most dimensions, the linear programming bound seems nowhere near sharp,
but the upper and lower bounds appear to touch in eight and twenty-four
dimensions. Cohn and Elkies conjectured that they were equal in those cases,
and the solutions of the sphere packing problem in these dimensions come from
proving this conjecture.\footnote{The linear programming bound also seems to
be sharp in two dimensions, but no proof is known, despite the fact that the
two-dimensional sphere packing problem itself can be solved by elementary
means.} For comparison, it is known that the linear programming bound cannot be sharp
in dimensions three through five \cite{L}, six \cite{dCDV}, twelve, or sixteen \cite{CdLS},
and it is likely that the only sharp cases are dimensions one, two, eight, and twenty-four.

The optimal auxiliary functions in eight and twenty-four dimensions have come
to be known as magic functions, because obtaining an exact solution in these
dimensions feels like a miracle. To see how this miracle comes about, we will
examine a proof of Theorem~\ref{thm:lpbound} for the special case of lattice
packings. It is based on the Poisson summation formula, which states that
\[
\sum_{x \in \Lambda} f(x) = \frac{1}{\vol(\R^n/\Lambda)} \sum_{y \in \Lambda^*} \widehat{f}(y)
\]
for every Schwartz function $f \colon \R^n \to \C$ and lattice $\Lambda$ in
$\R^n$. In this formula, $\vol(\R^n/\Lambda)$ is the volume of the quotient
torus (i.e., the volume of a fundamental parallelotope for the lattice or,
equivalently, the absolute value of the determinant of a basis), and
$\Lambda^*$ is the dual lattice, which is spanned by the dual basis
$v_1^*$, \dots, $v_n^*$ to any basis $v_1$, \dots, $v_n$ of $\Lambda$ (i.e., $\langle
v_i^*, v_j \rangle = \delta_{i,j}$). Poisson summation expresses a
fundamental duality for Fourier analysis on $\R^n$, and we can apply it as
follows.

\begin{proof}[Proof of Theorem~\ref{thm:lpbound} for lattice packings]
Suppose our sphere packing consists of spheres centered at the points of a
lattice $\Lambda$ in $\R^n$. The sphere packing density is scaling-invariant,
and so without loss of generality we can assume that the minimal nonzero
vectors in $\Lambda$ have length $r$. In other words, the sphere packing uses
spheres of radius $r/2$, so that neighboring spheres are tangent to each
other. Then the packing density is $\vol(B_{r/2}^n)/\!\vol(\R^n/\Lambda)$,
since there is one sphere for each fundamental cell of $\Lambda$.

We now apply Poisson summation to the auxiliary function $f$, to obtain
\[
\sum_{x \in \Lambda} f(x) = \frac{1}{\vol(\R^n/\Lambda)} \sum_{y \in \Lambda^*} \widehat{f}(y).
\]
The left side of this equation is bounded above by $f(0)=1$, because $f(x)
\le 0$ whenever $|x| \ge r$, and the right side is bounded below by
$\widehat{f}(0)/\!\vol(\R^n/\Lambda) = 1/\!\vol(\R^n/\Lambda)$, since every
summand is nonnegative. Thus, we conclude that $1/\!\vol(\R^n/\Lambda) \le
1$, and the sphere packing density satisfies
$\vol(B_{r/2}^n)/\!\vol(\R^n/\Lambda) \le \vol(B_{r/2}^n)$, as desired.
\end{proof}

The proof for more general packings is similar in spirit, but it applies
Poisson summation to periodic packings given by unions of translates of a
lattice. See \cite{CE} or \cite{C1} for the details.

Note that the proof of Theorem~\ref{thm:lpbound} does not actually require
$f$ to be radial. However, the conditions on $f$ are linear and
rotation-invariant, and thus we can assume $f$ is radial without loss of
generality via rotational averaging.

What sort of function $f$ could show that a lattice $\Lambda$ is an optimal
sphere packing? The proof given above drops the terms $f(x)$ with $x \in
\Lambda \setminus \{0\}$ and $\widehat{f}(t)$ for $y \in \Lambda^*
\setminus\{0\}$. Thus, we obtain a sharp bound if and only if all these
omitted terms vanish. Because $f$ and $\widehat{f}$ are radial functions,
these conditions amount to saying that $f$ vanishes on all the nonzero vector
lengths in $\Lambda$, while $\widehat{f}$ vanishes on all the nonzero vector
lengths in $\Lambda^*$. Furthermore, $f$ and $\widehat{f}$ cannot change sign
at these roots, except for a sign change in $f$ at the minimal nonzero vector
length in $\Lambda$.

It turns out that the $E_8$ and Leech lattices are both self-dual, and their
nonzero vector lengths are simply $\sqrt{2k}$ for integers $k \ge 1$ in $E_8$
and $k \ge 2$ in $\Lambda_{24}$. Thus, we know exactly what the roots of the
magic functions should be. These roots are shown in
Figure~\ref{figure:diagram} for eight dimensions.

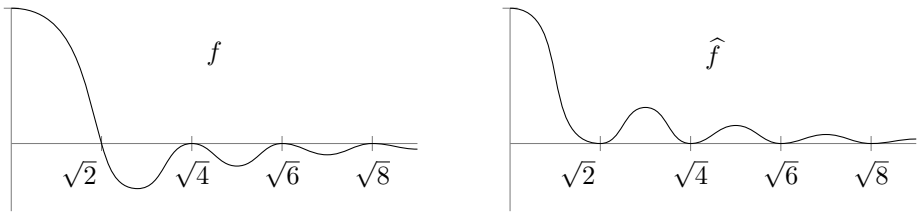
\begin{figure}
\centering
\begin{tikzpicture}[scale=1.2]
\draw [black!50] (0,-0.75) -- (0,1.5);
\draw [black!50] (-0.1/1.2,1.5) -- (0.1/1.2,1.5);
\draw [black!50] (0,0) -- (4.5,0);
\draw (2.25,1) node {$f$};
\draw [black!50] (1,-0.1/1.2) -- (1,0.1/1.2); \draw (0.75,-0.1/1.2) node[below] {$\sqrt{2}$};
\draw [black!50] (2,-0.1/1.2) -- (2,0.1/1.2); \draw (2,-0.1/1.2) node[below] {$\sqrt{4}$};
\draw [black!50] (3,-0.1/1.2) -- (3,0.1/1.2); \draw (3,-0.1/1.2) node[below] {$\sqrt{6}$};
\draw [black!50] (4,-0.1/1.2) -- (4,0.1/1.2); \draw (4,-0.1/1.2) node[below] {$\sqrt{8}$};
\draw (0,1.5) to[out=0,in=106] (1,0) to[out=286,in=180] (1.4,-0.5)
to[out=0,in=180] (2,0) to[out=0,in=180] (2.5,-0.25)
to[out=0,in=180] (3,0) to[out=0,in=180] (3.5,-0.125)
to[out=0,in=180] (4,0) to[out=0,in=180] (4.5,-0.0625);
\end{tikzpicture}
\hskip 1cm
\begin{tikzpicture}[scale=1.2]
\draw [black!50] (0,-0.75) -- (0,1.5);
\draw [black!50] (-0.1/1.2,1.5) -- (0.1/1.2,1.5);
\draw [black!50] (0,0) -- (4.5,0);
\draw (2.25,1) node {$\widehat{f}$};
\draw [black!50] (1,-0.1/1.2) -- (1,0.1/1.2); \draw (0.75,-0.1/1.2) node[below] {$\sqrt{2}$};
\draw [black!50] (2,-0.1/1.2) -- (2,0.1/1.2); \draw (2,-0.1/1.2) node[below] {$\sqrt{4}$};
\draw [black!50] (3,-0.1/1.2) -- (3,0.1/1.2); \draw (3,-0.1/1.2) node[below] {$\sqrt{6}$};
\draw [black!50] (4,-0.1/1.2) -- (4,0.1/1.2); \draw (4,-0.1/1.2) node[below] {$\sqrt{8}$};
\draw (0,1.5) to[out=0,in=180] (1,0) to[out=0,in=180] (1.5,0.4)
to[out=0,in=180] (2,0) to[out=0,in=180] (2.5,0.2)
to[out=0,in=180] (3,0) to[out=0,in=180] (3.5,0.1)
to[out=0,in=180] (4,0) to[out=0,in=180] (4.5,0.05);
\end{tikzpicture}
\caption{This diagram, which is taken from \cite{C1}, shows
the roots of the magic function $f$ and its Fourier transform
$\widehat{f}$ in eight dimensions. It is not an accurate plot, since
these functions decrease very rapidly.}
\label{figure:diagram}
\end{figure}

Now the whole problem comes down to constructing magic functions with these
roots. That might not seem so difficult, but controlling the behavior of $f$
and $\widehat{f}$ simultaneously is a subtle problem. Of course we can obtain
any roots we'd like for $f$ or $\widehat{f}$ in isolation, but not
necessarily at the same time. This phenomenon is a form of uncertainty
principle \cites{BCK,GOS,CG}, much like the Heisenberg uncertainty principle.

Viazovska gave a remarkable construction of the eight-dimensional magic
function in terms of modular forms, which are a class of special functions
defined on the upper half-plane $\Hyp = \{ z \in \C : \Im(z) > 0\}$ and
satisfying certain transformation laws. The general theory of modular forms
can feel somewhat forbidding to beginners, but Poisson summation gives us a
simple way to get our hands on one example. The theta function $\theta \colon
\Hyp \to \C$ is defined by
\[
\theta(z) = \sum_{n \in \Z} e^{\pi i n^2 z} = 1 + 2 e^{\pi i z} + 2 e^{4 \pi i z} + 2 e^{9 \pi i z} + \cdots,
\]
which converges because $z \in \Hyp$ means $\Im(z)>0$ and thus $|e^{\pi i z}|<1$. This function
satisfies two key identities,
\begin{equation} \label{eq:thetatransformations}
\theta(z+2) = \theta(z) \qquad\text{and}\qquad \theta(-1/z) = (-i z)^{1/2} \theta(z).
\end{equation}
The first identity follows immediately from the defining series, while the
second is more subtle and will be proved below. In this equation, we have to choose
the branch for $(-i z)^{1/2}$ carefully. Throughout this paper, fractional powers such as
this one will be defined to be positive on the upper imaginary axis $(0,
\infty) i$ in $\Hyp$ and continuous on $\Hyp$.

To prove that $\theta(-1/z) = (-i z)^{1/2} \theta(z)$, we will use Poisson
summation for the one-dimensional lattice $\Z$ in $\R$. Consider the complex
Gaussian $f \colon \R \to \C$ defined by
\[
f(x) = e^{\pi i z x^2}
\]
with $z \in \Hyp$. When $z$ is purely imaginary, this function is an ordinary
Gaussian, and the other points in $\Hyp$ behave much the same. In particular,
one can check that
\begin{equation}
\label{eq:complexGaussian}
\widehat{f}(y) = (- i z)^{-1/2} e^{\pi i (-1/z) y^2},
\end{equation}
which is the complex generalization of the fact that the Fourier transform of
a wide Gaussian is a narrow Gaussian, and vice versa. Now Poisson summation
says that
\[
\sum_{x \in \Z} f(x) = \sum_{y \in \Z} \widehat{f}(y),
\]
because $\Z$ is self-dual. This equation amounts to
\[
\sum_{x \in \Z} e^{\pi i z x^2} = \sum_{y \in \Z} (- i z)^{-1/2} e^{\pi i (-1/z) y^2},
\]
and thus $\theta(-1/z) = (-i z)^{1/2} \theta(z)$.

The functions $z \mapsto z+2$ and $z \mapsto -1/z$ map $\Hyp$ to itself, and
they generate a group of linear fractional transformations of $\Hyp$ called
$\Gamma_\theta$, in honor of the function~$\theta$.  One can put a metric on
$\Hyp$ that turns it into the hyperbolic plane, at which point
$\Gamma_\theta$ becomes a discrete group of isometries of $\Hyp$, but we will
not need this interpretation. See Figure~\ref{fig:poles} for a picture of a
$\Gamma_\theta$-orbit in $\Hyp$.

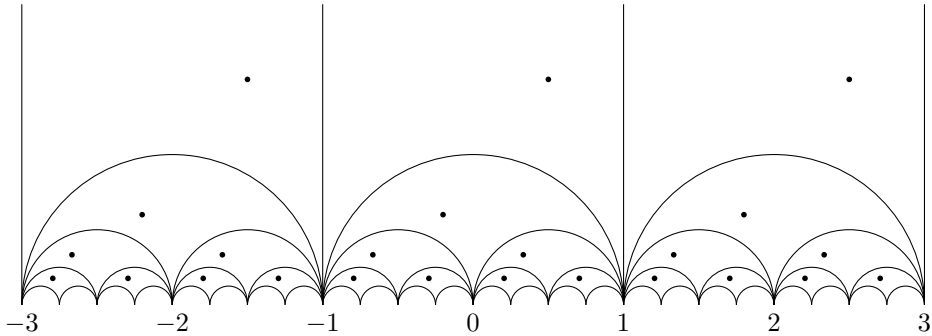
\begin{figure}
\centering
\begin{tikzpicture}
\draw (-6,0) node[below] {$-3$};
\draw (-4,0) node[below] {$-2$};
\draw (-2,0) node[below] {$-1$};
\draw (0,0) node[below] {$0$};
\draw (2,0) node[below] {$1$};
\draw (4,0) node[below] {$2$};
\draw (6,0) node[below] {$3$};
\foreach \i in {-4,-2,...,6}
{
\draw (\i,0) arc (0:180:1);
}
\foreach \i in {-2,2,6}
{
\draw (\i,0) arc (0:180:2);
}
\foreach \i in {-5,-4,...,6}
{
\draw (\i,0) arc (0:180:{1/2});
}
\foreach \i in {-11,-10,...,12}
{
\draw ({\i/2},0) arc (0:180:{1/4});
}
\draw (6,0)--(6,4);
\draw (2,0)--(2,4);
\draw (-2,0)--(-2,4);
\draw (-6,0)--(-6,4);
\fill (1,3) circle (0.0375);
\fill (-3,3) circle (0.0375);
\fill (5,3) circle (0.0375);
\fill (-0.4,1.2) circle (0.0375);
\fill (-4.4,1.2) circle (0.0375);
\fill (3.6,1.2) circle (0.0375);
\fill ({2/3},{2/3}) circle (0.0375);
\fill ({2/3-2},{2/3}) circle (0.0375);
\fill ({2/3-4},{2/3}) circle (0.0375);
\fill ({2/3-6},{2/3}) circle (0.0375);
\fill ({2/3+2},{2/3}) circle (0.0375);
\fill ({2/3+4},{2/3}) circle (0.0375);
\foreach \i in {-5,-4,...,6}
\fill ({\i-10/17},{6/17})  circle (0.0375);
\end{tikzpicture}
\caption{The regions shown here are ideal hyperbolic triangles (i.e., triangles
in the hyperbolic plane with vertices at infinity), and they are
fundamental domains for the action of $\Gamma_\theta$ on the upper half-plane.
In particular, each $\Gamma_\theta$-orbit intersects each triangle exactly once,
unless it intersects the boundary of the triangle. The dots show
a typical $\Gamma_\theta$-orbit.}
\label{fig:poles}
\end{figure}

Together with analyticity and some growth conditions, the identities
\eqref{eq:thetatransformations} say that $\theta$ is a \emph{modular form of
weight $1/2$} for the group $\Gamma_\theta$. Viazovska's solution of the
eight-dimensional sphere packing problem constructs the magic function using
$\theta$ and a number of other modular forms, in a way that looks rather
mysterious. What do modular forms have to do with radial
Schwartz functions?

Instead of examining the details of Viazovska's construction, let's think about a
bigger picture. We know the eight-dimensional magic function $f$ should
satisfy
\begin{align*}
f\big(\sqrt{2k}\big) &= 0 \qquad \text{for $k \ge 1$,}\\
f'\big(\sqrt{2k}\big) &= 0 \qquad \text{for $k \ge 2$,}\\
\widehat{f}\big(\sqrt{2k}\big) &=0 \qquad \text{for $k \ge 1$, and }\\
\widehat{f}\,'\big(\sqrt{2k}\big) & = 0 \qquad \text{for $k \ge 1$,}
\end{align*}
as in Figure~\ref{figure:diagram}. Viazovska conjectured that this data,
together with the nonzero value $f'\big(\sqrt{2}\big)$, would be enough to
determine $f$ uniquely. In fact, that turns out to be true:

\begin{theorem}[Cohn, Kumar, Miller, Radchenko, and Viazovska \cite{CKMRV2}]
\label{theorem:interpolatedouble} Let $(n,k_0)$ be $(8,1)$ or $(24,2)$. Then
every radial Schwartz function $f \colon \R^n \to \C$ is uniquely determined
by the values $f\big(\sqrt{2k}\big)$, $f'\big(\sqrt{2k}\big)$,
$\widehat{f}\big(\sqrt{2k}\big)$, and $\widehat{f}\,'\big(\sqrt{2k}\big)$ for
integers $k \ge k_0$.  Specifically, there exists an interpolation basis
$a_k, b_k, \widehat{a}_k, \widehat{b}_k$ of radial Schwartz functions on
$\R^n$ for $k \ge k_0$ such that for every $f$ and $x \in \R^n$,
\[
\begin{split}
f(x) &= \sum_{k=k_0}^\infty f\big(\sqrt{2k}\big)\,a_k(x)+\sum_{k=k_0}^\infty
f'\big(\sqrt{2k}\big)\,b_k(x)\\
&\quad\phantom{}+\sum_{k=k_0}^\infty \widehat{f}\big(\sqrt{2k}\big)\,\widehat{a}_k(x)
+\sum_{k=k_0}^\infty \widehat{f}\,'\big(\sqrt{2k}\big)\,\widehat{b}_k(x),
\end{split}	
\]
where these sums converge absolutely.
\end{theorem}

In particular, up to scaling the magic function is the interpolation basis
function $b_{k_0}$ in this theorem. One does not need this interpolation
theorem to solve the sphere packing problem, but it is needed for analyzing
ground states of more general particle systems in $\R^8$ and $\R^{24}$ (see
\cite{CKMRV2}), and it provides a broader context for the magic functions.

Theorem~\ref{theorem:interpolatedouble} is similar in spirit to other
interpolation theorems in mathematics. The simplest and most famous of these
theorems is Lagrange interpolation, which says that a polynomial in one
variable of degree less than $n$ can be reconstructed from its values at any
$n$ distinct points. If the interpolation points are $x_1$, \dots, $x_n$, then we
can write down an interpolation basis $p_1$, \dots, $p_n$ as
\begin{equation} \label{eq:lagrangeproduct}
p_k(x) = \prod_{\substack{j=1\\j \ne k}}^n \frac{x-x_j}{x_k-x_j},
\end{equation}
so that every polynomial $f$ of degree less than $n$ is given by
\[
f(x) = \sum_{j=1}^n f(x_j) p_j(x).
\]
Lagrange interpolation can be generalized to Hermite interpolation, which
takes into account derivatives along similar lines to
Theorem~\ref{theorem:interpolatedouble}: a polynomial $f$ can be
reconstructed from the values $f^{(j)}(x_k)$ with $0 \le j < d_k$ and $1 \le
k \le m$ if its degree is less than $\sum_{k=1}^m d_k$.

\addtolength{\textheight}{1.25\baselineskip} 

One important relative of Lagrange interpolation is Shannon sampling, which
in the case of Schwartz functions $f \colon \R \to \C$ says that if
$\widehat{f}$ vanishes outside the interval $[-r/2,r/2]$ for some $r$, then
$f$ is determined by its values on $r^{-1} \Z$ via
\[
f(x) = \sum_{n \in \Z} f(n/r) \frac{\sin \pi (rx-n)}{\pi (rx-n)}.
\]
This theorem plays a crucial role in information theory, since it says that a
band\-limited signal (i.e., one with a limited range of frequencies) is
determined by periodic samples. It's worth noting that the product formula
\begin{equation} \label{eq:sineproduct}
\frac{\sin \pi x}{\pi x} = \prod_{j=1}^\infty \left( 1-\frac{x}{j^2}\right)
\end{equation}
is analogous to the products \eqref{eq:lagrangeproduct} in the Lagrange
interpolation basis. Much is known about Shannon sampling and its variations;
see, for example, \cite{Hig} and the references cited therein.

Both Lagrange interpolation and Shannon sampling rely on a notion of size. We
measure the size of a polynomial by its degree, and the size of a
band-limited function by its bandwidth, the smallest $r$ such that
$\supp(\widehat{f}) \subseteq [-r/2,r/2]$. Then the larger a function is, the
more interpolation points are required to reconstruct it, with ``more''
referring to density in the band-limited case. Here the intuition is that size
controls how many roots a function can have.\footnote{Furthermore, size is
related to growth at infinity. For degrees of polynomials this is clear,
while a band-limited function of bandwidth $r$ can be analytically continued
to the entire complex plane and satisfies $|f(z)| = O(e^{\pi r |z|})$. In
other words, it is an entire function of exponential type~$\pi r$.}

Puzzlingly, Theorem~\ref{theorem:interpolatedouble} shows no sign of a
similar notion of size. It is reminiscent of Shannon sampling, in that it
takes into account both $f$ and $\widehat{f}$, but it treats them
symmetrically. In particular, there is little hope of a product formula along
the lines of \eqref{eq:lagrangeproduct} or \eqref{eq:sineproduct}, because
specifying the roots of $f$ will not yield the roots of $\widehat{f}$. There
seems to be a fundamental difference between these interpolation formulas,
and neither Lagrange interpolation nor Shannon sampling offers a clue
as to how to prove Theorem~\ref{theorem:interpolatedouble}.

\section{First-order Fourier interpolation}

How does one prove an interpolation theorem like
Theorem~\ref{theorem:interpolatedouble}? We'll examine a technically simpler
variant due to Radchenko and Viazovska, which is important in its own right
and a beautiful illustration of Fourier interpolation. It deals with
functions of one variable (so ``radial'' becomes ``even''), and it studies
interpolation to first order, without derivatives. This first-order
interpolation theorem does not seem to have any applications to sphere
packing, but it's a fundamental fact about Fourier analysis, and it is
remarkable that it was not known until well into the 21st century.

\begin{theorem}[Radchenko and Viazovska \cite{RV}]
\label{theorem:interpolateR1} There exist even Schwartz functions $a_n \colon
\R \to \R$ for $n = 0$, $1$, $2$, \dots\  such that every even Schwartz function $f
\colon \R \to \R$ satisfies
\[
f(x) = \sum_{n=0}^\infty f\big(\sqrt{n}\big) \, a_n(x) + \sum_{n=0}^\infty \widehat{f}\big(\sqrt{n}\big) \, \widehat{a}_n(x)
\]
for all $x \in \R$, and these sums converge absolutely.
\end{theorem}

\addtolength{\textheight}{-1.25\baselineskip} 

There is also a corresponding theorem about odd functions 
\cite[Theorem~7]{RV}, which can be proved in almost the same way. We'll focus on even
functions here for simplicity. Note also that the root spacing has changed
from $\sqrt{2n}$ to $\sqrt{n}$ in comparison with
Theorem~\ref{theorem:interpolatedouble}, which reflects the change in the
order of interpolation.

As a consequence of this interpolation theorem, if an even Schwartz function
$f \colon \R \to \R$ satisfies $f\big(\sqrt{n}\big) =
\widehat{f}\big(\sqrt{n}\big) = 0$ for $n=0$, $1$, $2$, \dots, then $f$ vanishes
identically. It's not so surprising that constructing an explicit
interpolation basis\break 
$a_0$, $a_1$, \dots\ would require special functions, such as
modular forms, but it's note\-worthy that even this corollary about vanishing
does not seem easy to prove directly.

In the remainder of this section, we'll sketch a proof of
Theorem~\ref{theorem:interpolateR1}. The sketch will omit a number of
analytic details, but it will outline the techniques and explain where
additional work is required.

The central question is where the interpolation basis $a_0$, $a_1$, \dots\  comes
from. We need to characterize these functions and prove that they have the
desired properties. A first observation is that the interpolation basis is
not quite unique, because Poisson summation over $\Z$ implies that every even
Schwartz function $f$ satisfies
\[
f(0) + 2f(1) + 2f(2) + \cdots = \widehat{f}(0) + 2\widehat{f}(1) + 2\widehat{f}(2) + \cdots.
\]
In particular, $\widehat{f}(0)$ is determined by the values
$f(0)$, $f(1)$, $f(2)$, \dots\  and $\widehat{f}(1)$, $\widehat{f}(2)$, \dots. To
account for this redundancy, we will impose the constraint $\widehat{a}_0 =
a_0$, so that the interpolation formula becomes
\[
f(x) = (f(0) + \widehat{f}(0))a_0(x) +  \sum_{n=1}^\infty f\big(\sqrt{n}\big) \, a_n(x) + \sum_{n=1}^\infty \widehat{f}\big(\sqrt{n}\big) \, \widehat{a}_n(x).
\]
It turns out that this formula is now irredundant, with no additional linear
relations between the values $f\big(\sqrt{n}\big)$ and
$\widehat{f}\big(\sqrt{n}\big)$, and the interpolation basis is uniquely
determined. Substituting $f = a_n$ shows that we can characterize $a_n$ by
its values at the points $\sqrt{m}$ with $m=0$, $1$, $2$, \dots. Specifically, for
$n,m \ge 1$, we must have
\[
a_n\big(\sqrt{m}\big) = \begin{cases}
1 & \text{if $m=n$, and}\\
0 & \text{otherwise,}
\end{cases}
\]
$\widehat{a}_n\big(\sqrt{m}\big) = 0$, and $a_n(0) + \widehat{a}_n(0)=0$,
while $a_0$ must satisfy $\widehat{a}_0 = a_0$, $a_0(0) = 1/2$, and
$a_0\big(\sqrt{m}\big)=0$ for all $m \ge 1$.

These constraints let us get a handle on $a_n$, and we can use them to
compute numerical approximations to $a_n$. More dramatically, they allow us
to use Viazovska's modular form techniques from \cite{V} to construct $a_n$
explicitly. For example, we can write down $a_0$ as follows:

\begin{lemma} \label{lemma:a0}
Let $a_0 \colon \R \to \C$ be defined by
\[
a_0(x) = \frac{1}{4} \int_{-1}^1 \theta(z)^3 e^{\pi i z x^2} \, dz,
\]
where we integrate over a semicircle in the upper half-plane $\Hyp$. Then
$a_0$ is an even Schwartz function with Fourier transform $\widehat{a}_0 =
a_0$, and it satisfies $a_0(0)=1/2$ and $a_0\big(\sqrt{m}\big) = 0$ for all
positive integers $m$.
\end{lemma}

We'll use the same semicircular contour of integration in all integrals from
$-1$ to~$1$ below. Recall that the theta function in this integral is defined
for $z \in \Hyp$ by
\[
\theta(z) = \sum_{n \in \Z} e^{\pi i n^2 z}
\]
and satisfies the functional equations $\theta(z+2) = \theta(z)$ and
$\theta(-1/z) = (-iz)^{1/2} \theta(z)$.

\begin{proof}[Sketch of proof]
The function $a_0$ is manifestly even, and we can prove that it is a Schwartz
function by analyzing the behavior of $\theta(z)$ as $z$ tends to $\pm 1$.
Specifically, if we remove small neighborhoods of $\pm 1$ from the contour,
then we obtain a smooth function of $x$. One can show that this function and its
derivatives are rapidly decreasing as $x \to \infty$, essentially because the complex
phases interfere destructively. To show that $a_0$ itself is a Schwartz
function, we just have to check that the behavior as $z \to \pm 1$ is not bad
enough to ruin this analysis. We will omit the details here.

To show that $\widehat{a}_0 = a_0$, we can take the Fourier transform of the
complex Gaussian under the integral sign using \eqref{eq:complexGaussian} and
change variables to $u = -1/z$, to obtain
\begin{align*}
\widehat{a}_0(x) &=  \frac{1}{4} \int_{-1}^1 \theta(z)^3 (-iz)^{-1/2} e^{\pi i (-1/z) x^2} \, dz\\
&= \frac{1}{4} \int_{1}^{-1} \theta(-1/u)^3 (i/u)^{-1/2} e^{\pi i u x^2} u^{-2} \, du\\
&= \frac{1}{4} \int_{-1}^1 \theta(u)^3 e^{\pi i u x^2} \, du\\
&= a_0(x),
\end{align*}
where the third line follows from $\theta(-1/u)^3 = (-iu)^{3/2} \theta(u)^3$ and
\[
-(-i u)^{3/2} (i/u)^{-1/2} u^{-2} = 1
\]
for $u \in \Hyp$. (To check this last identity, note that the left side is
always $\pm 1$, it is continuous for $u \in \Hyp$, and it equals $1$ when
$u=i$.)

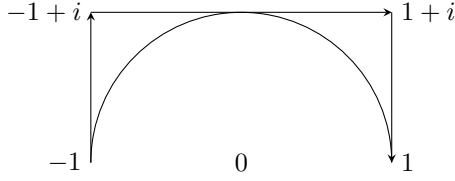
\begin{figure}
\centering
\begin{tikzpicture}[scale=2]
\draw[-stealth] (-1,0)--(-1,1);
\draw[-stealth] (-1,1)--(1,1);
\draw[-stealth] (1,1)--(1,0);
\draw (1,0) arc (0:180:1);
\draw (-1,0) node[left] {$-1$};
\draw (-1,1) node[left] {$-1+i$};
\draw (1,0) node[right] {$1$};
\draw (1,1) node[right] {$1+i$};
\draw (0,0) node {$0$};
\end{tikzpicture}
\caption{When we deform the semicircle into this polygonal path, the vertical
sides cancel because $\theta(z+2)=\theta(z)$.}
\label{fig:defcircline}
\end{figure}

Finally, we can compute $a_0\big(\sqrt{m}\big)$ for nonnegative integers $m$
using the identity
\begin{align*}
a_0\big(\sqrt{m}\big) &= \frac{1}{4} \int_{-1}^1 \theta(z)^3 e^{m \pi i z} \, dz\\
&= \frac{1}{4} \int_{-1+i}^{1+i}  \theta(z)^3 e^{m \pi i z} \, dz,
\end{align*}
where we have deformed the contour to a straight line from $-1+i$ to $1+i$ as in Figure~\ref{fig:defcircline},
which is possible because the integrals between $0$ and $-1+i$ and between
$1+i$ and $1$ cancel due to $\theta(z+2)=\theta(z)$. Now we write
\[
\theta(z) = 1 + 2 e^{\pi i z} + 2 e^{4 \pi i z} + 2 e^{9 \pi i z} + \cdots
\]
and expand $\theta(z)^3$ as a series in powers of $e^{\pi i z}$. By Fourier
orthogonality, the value
\[
a_0\big(\sqrt{m}\big) = \frac{1}{4} \int_{-1+i}^{1+i}  \theta(z)^3 e^{m \pi i z} \, dz
\]
is $1/2$ times the coefficient of $e^{-m \pi i z}$ in this expansion of
$\theta(z)^3$. In particular, $a_0(0)=1/2$ and $a_0\big(\sqrt{m}\big)=0$ for
positive integers $m$, as desired, since there are no negative powers of
$e^{\pi i z}$ in this series.
\end{proof}

What made this proof work is that the identity $\theta(-1/z)^3 = (-iz)^{3/2}
\theta(z)^3$ gave us $\widehat{a}_0=a_0$, while the identity $\theta(z+2)^3 =
\theta(z)^3$ let us compute the values $a_0\big(\sqrt{m}\big)$ as Fourier
series coefficients. One can obtain each basis function $a_n$ using similar
constructions, which require increasingly elaborate replacements for
$\theta(z)^3$ as $n$ grows, and it is not immediately clear how to describe
or analyze them systematically. Furthermore, obtaining the basis functions
individually does not explain why the interpolation formula actually holds:
these functions could in principle exist yet not suffice to reconstruct an
arbitrary even Schwartz function in Theorem~\ref{theorem:interpolateR1}.

To give a uniform account of these functions, we will construct generating
functions for the interpolation basis. For $\tau \in \Hyp$, let
\[
F(\tau, x) = \sum_{n=0}^\infty  a_n(x) e^{n \pi i \tau},
\]
and denote its Fourier transform in $x$ by
\[
\widehat{F}(\tau,x) = \sum_{n=0}^\infty \widehat{a}_n(x) e^{n \pi i  \tau}.
\]
Being Fourier series, these functions satisfy the functional equations
\begin{equation} \label{eq:add2}
F(\tau+2,x) = F(\tau,x) \qquad\text{and}\qquad \widehat{F}(\tau+2,x) = \widehat{F}(\tau,x).
\end{equation}
Furthermore, formula \eqref{eq:complexGaussian} for the Fourier transform
of a complex Gaussian implies that the interpolation formula from
Theorem~\ref{theorem:interpolateR1} for the function $f(x) =  e^{\pi i \tau
x^2}$ is equivalent to
\[
F(\tau,x) + (-i \tau)^{-1/2} \widehat{F}(-1/\tau,x) = e^{\pi i \tau x^2},
\]
and thus $F$ and $\widehat{F}$ must satisfy this functional equation, in addition to
those in \eqref{eq:add2}.

In fact, these three functional equations turn out to be almost all we need to
obtain a working interpolation basis. Lemma~\ref{lemma:reduction} is stated somewhat
informally, but it can be made precise.

\begin{lemma} \label{lemma:reduction}
If there exists a function $F$ such that $F$ and $\widehat{F}$ satisfy these
three functional equations and certain analyticity and growth bounds, then
Theorem~\ref{theorem:interpolateR1} follows.
\end{lemma}

\begin{proof}[Sketch of proof]
The idea behind the proof is surprisingly simple. If $F$ and $\widehat{F}$
are sufficiently well-behaved, then the functional equations $F(\tau+2,x) =
F(\tau,x)$ and $\widehat{F}(\tau+2,x) = \widehat{F}(\tau,x)$ imply that they
can be expanded as Fourier series. We can define the functions $a_n$ to be
the Fourier coefficients of $F(\tau,x)$, and $\widehat{a}_n$ must be the
corresponding coefficient of $\widehat{F}(\tau, x)$, as in the original
definitions of $F$ and $\widehat{F}$ above. The fact that there are no terms
with $n<0$ amounts to boundedness as $\Im(\tau) \to \infty$, and the
constraint that $a_0 = \widehat{a}_0$ can be phrased similarly (namely that
$F(\tau,x)-\widehat{F}(\tau,x)$ decays as $\Im(\tau) \to \infty$).

Now the third functional equation says that
\[
 \sum_{n=0}^\infty a_n(x) e^{n \pi i  \tau} +  \sum_{n=0}^\infty \widehat{a}_n(x) (-i \tau)^{-1/2} e^{n \pi i  (-1/\tau)} = e^{\pi i \tau x^2},
\]
which becomes
\[
\sum_{n=0}^\infty a_n(x) f\big(\sqrt{n}\big) + \sum_{n=0}^\infty \widehat{a}_n(x) \widehat{f}\big(\sqrt{n}\big) = f(x)
\]
if we set $f(x) = e^{\pi i \tau x^2}$. In other words, it states that the
interpolation theorem holds when $f$ is a complex Gaussian.

One can show that complex Gaussians span a dense subspace of the even
Schwartz functions. To complete the proof, all we need to show is that for
each $x \in \R$, the functional $\Lambda_x$ that takes an even Schwartz
function $f$ to
\[
\Lambda_x(f) = f(x) - \sum_{n=0}^\infty f\big(\sqrt{n}\big) \, a_n(x) - \sum_{n=0}^\infty \widehat{f}\big(\sqrt{n}\big) \, \widehat{a}_n(x)
\]
is continuous, so that vanishing on a dense subspace implies vanishing
everywhere. The topology on the space of Schwartz functions is defined by a
family of seminorms, and proving that $\Lambda_x$ is continuous requires
proving that the seminorms of $a_n$ and $\widehat{a}_n$ grow at most
polynomially as $n \to \infty$. To prove the required bounds, we can use
Fourier orthogonality to write $a_n(x)$ and $\widehat{a}_n(x)$ as integrals
in $\tau$ of $F(\tau,x)$ and $\widehat{F}(\tau,x)$, respectively, and then
use suitable growth bounds for $F$ and $\widehat{F}$ to bound the seminorms
of these integrals.
\end{proof}

We can now imitate the construction of $a_0$ from $\theta(z)^3$ in
Lemma~\ref{lemma:a0} to obtain the generating functions $F$ and $\widehat{F}$
explicitly. To do so, we will replace $\theta(z)^3$ with the functions $K$
and $\widehat{K}$ from Proposition~\ref{prop:kernels}, which is again stated
informally. Note that $\widehat{K}$ is not a Fourier transform of $K$;
instead, this notation is simply mnemonic, since $\widehat{K}$ will be used
to construct $\widehat{F}$.

\begin{proposition} \label{prop:kernels}
There exist meromorphic functions $K$ and $\widehat{K}$ on $\Hyp \times \Hyp$
that satisfy the following conditions for all $\tau,z \in \Hyp$:
\begin{enumerate}
\item \label{item1} $K(\tau+2,z) = K(\tau,z)$ and $\widehat{K}(\tau+2,z) = \widehat{K}(\tau,z)$,

\item \label{item2} $K(\tau,z+2) = K(\tau,z)$ and $\widehat{K}(\tau,z+2) = \widehat{K}(\tau,z)$,

\item \label{item3} $K(-1/\tau,z) = -(-i \tau)^{1/2} \widehat{K}(\tau,z)$,

\item \label{item4} $K(\tau,-1/z) = (-i z)^{3/2} \widehat{K}(\tau, z)$,

\item \label{item5} $z \mapsto K(\tau,z)$ and $z \mapsto
    \widehat{K}(\tau,z)$ have poles only when $z$ is in the
    $\Gamma_\theta$-orbit\break 
    of $\tau$,

\item \label{item6} all their poles are simple poles,

\item \label{item7} the residue of $z \mapsto K(\tau,z)$ at $z=\tau$ is $1/(2\pi i)$ and at $z=-1/\tau$ is $0$ \textup{(}in other words, there is no pole there\textup{)}, 

\item \label{item8} the residue of $z \mapsto \widehat{K}(\tau,z)$ at $z=\tau$ is $0$, and

\item \label{item9} $K$ and $\widehat{K}$ satisfy certain growth bounds, which we will not discuss here.
\end{enumerate}
\end{proposition}

The motivation behind the transformation laws in
Proposition~\ref{prop:kernels} is that they generalize how $\theta(z)^3$
transforms, and we'll see that they perfectly describe what we need to obtain
$F$ and $\widehat{F}$ as integrals of $K$ and $\widehat{K}$. At first glance
the most mysterious aspect may be the poles, which did not occur for
$\theta(z)^3$. We'll see below that the poles lead to the inhomogeneous term
$e^{\pi i \tau x^2}$ in the functional equation
\[
F(\tau,x) + (-i \tau)^{-1/2} \widehat{F}(-1/\tau,x) = e^{\pi i \tau x^2}.
\]

Before we examine how to use $K$ and $\widehat{K}$ to construct $F$ and
$\widehat{F}$, we will take a look at how Proposition~\ref{prop:kernels} is
proved.

\begin{proof}[Sketch of proof]
The functions $K$ and $\widehat{K}$ can be described explicitly in terms of
modular forms, using three ingredients: the theta function $\theta$, the
modular function $\lambda$, and a Hauptmodul (principal modular function) $J$
for the group $\Gamma_\theta$.

We have already been using $\theta$, and $\lambda$ is a similar analytic
function on $\Hyp$ that dates back to the 19th century. For our purposes, its
key properties will be how $\Gamma_\theta$ acts on~it, namely
\[
\lambda(z+2) = \lambda(z) \qquad\text{and}\qquad \lambda(-1/z) = 1-\lambda(z).
\]
Note that it is not quite invariant under $\Gamma_\theta$. We define $J(z)$
to be $\lambda(z)(1-\lambda(z))/16$, so that $J(z)$ is invariant under both
generators of $\Gamma_\theta$; i.e.,
\[
J(z+2) = J(z) \qquad\text{and}\qquad J(-1/z) = J(z).
\]
Then it turns out that $J$ generates the function field of the quotient of
$\Hyp$ by the action of $\Gamma_\theta$ (this quotient has genus~$0$), and
$J(z) = J(\tau)$ if and only if $z$ and $\tau$ are in the same orbit of
$\Gamma_\theta$.

Using these tools, we can guess much of what $K(\tau,z)$ and
$\widehat{K}(\tau,z)$ should look like. Conditions~\eqref{item3}
and~\eqref{item4} suggest that these functions should have factors of
$\theta(\tau) \theta(z)^3$ to get the correct weights for the transformation
laws. Conditions~\eqref{item4} and~\eqref{item5} imply that they should be
given by $1/(J(z)-J(\tau))$ times something holomorphic, and the signs in
\eqref{item3} and~\eqref{item4} can be obtained using $1 - 2\lambda(-1/z) =
-(1-2\lambda(z))$.

In fact, we can take
\[
K(\tau,z) = \theta(\tau) \theta(z)^3 \frac{J(z)(1-2\lambda(\tau)) + J(\tau)(1-2\lambda(z))}{4(J(z) - J(\tau))}
\]
and
\[
\widehat{K}(\tau,z) = \theta(\tau) \theta(z)^3 \frac{J(z)(1-2\lambda(\tau)) - J(\tau)(1-2\lambda(z))}{4(J(z) - J(\tau))},
\]
and fairly routine computations show that conditions \eqref{item1} through~\eqref{item9}
hold. The functions $K$ and $\widehat{K}$ turn out to be uniquely determined
by these conditions, but we will not verify that here, to avoid having to
state the conditions more carefully and deal with residues and growth bounds.

It's worth noting that one can simplify some of the verification by writing
$K$ and $\widehat{K}$ in terms of the function $h := 1-2\lambda$ via
\[
K(\tau,z) = \theta(\tau) \theta(z)^3 \frac{1-h(\tau) h(z)}{4(h(\tau)-h(z))}
\]
and
\[
\widehat{K}(\tau,z) = \theta(\tau) \theta(z)^3 \frac{1+h(\tau) h(z)}{4(h(\tau)+h(z))}.
\]
For example, $h$ is a Hauptmodul for a subgroup of $\Gamma_\theta$ called
$\Gamma(2)$, and these formulas show that the poles of $z \mapsto K(\tau,z)$
and $z \mapsto \widehat{K}(\tau,z)$ occur only on the $\Gamma(2)$-orbits of
$\tau$ and $-1/\tau$, respectively.
\end{proof}

All that remains is to use $K$ and $\widehat{K}$ to construct functions $F$
and $\widehat{F}$ for use in Lemma~\ref{lemma:reduction}. To do so, we can
imitate Lemma~\ref{lemma:a0}. As a first attempt to produce $F$ from $K$, we
could try setting
\begin{equation} \label{eq:Fdef}
F(\tau,x) = \int_{-1}^1 K(\tau, z) e^{\pi i z x^2} \, dz.
\end{equation}
However, this formula can't possibly hold for all $\tau$, because the
integrand has poles on the $\Gamma_\theta$-orbit of $\tau$, and as one varies
$\tau$, sometimes these poles cross the contour of integration. Instead, we
can use this definition only on subsets of $\Hyp$ for which the poles avoid
the contour of integration. As shown in Figure~\ref{fig:poles}, one such
subset consists of all the points $\tau \in \Hyp$ such that $\tau$ has
distance strictly greater than~$1$ from~$2\Z$. For such $\tau$, we define
$F(\tau,x)$ by \eqref{eq:Fdef}; we will deal with other values of $\tau$ via
analytic continuation in Lemma~\ref{lemma:analyticcont}.

To obtain $\widehat{F}(\tau,x)$ we can take the Fourier transform of
$F(\tau,x)$ in $x$. For $\tau$ strictly further than distance $1$ from $2\Z$,
we can use the semicircular contour, and almost exactly the same computation
as in the proof of Lemma~\ref{lemma:a0} shows that
\begin{align*}
\widehat{F}(\tau,x) &= \int_{-1}^1 K(\tau, z) (-i z)^{-1/2} e^{\pi i (-1/z) x^2} \, dz\\
&= -  \int_{-1}^1 K(\tau,-1/z) (i/z)^{-1/2} e^{\pi i z x^2} z^{-2} \, dz\\
 &= -  \int_{-1}^1 \widehat{K}(\tau, z) (-i z)^{3/2} (i/z)^{-1/2}z^{-2} e^{\pi i z x^2} \, dz\\
 &=  \int_{-1}^1 \widehat{K}(\tau, z) e^{\pi i z x^2} \, dz.
\end{align*}

\begin{lemma} \label{lemma:analyticcont}
The functions $\tau \mapsto F(\tau,x)$ and $\tau \mapsto \widehat{F}(\tau,x)$
can be analytically continued to all of $\Hyp$, and they satisfy the
functional equations $F(\tau+2,x) = F(\tau,x)$, $\widehat{F}(\tau+2,x) =
\widehat{F}(\tau,x)$, and $F(\tau,x) + (-i \tau)^{-1/2}
\widehat{F}(-1/\tau,x) = e^{\pi i \tau x^2}$.
\end{lemma}

\begin{proof}[Sketch of proof]
Let $S = \{ \tau \in \Hyp: |\tau-2n| > 1 \text{ for all $n \in \Z$}\}$. We
have defined $F(\tau,x)$ and $\widehat{F}(\tau,x)$ for $\tau \in S$, and the
functional equations
\[
F(\tau+2,x) = F(\tau,x) \qquad\text{and}\qquad \widehat{F}(\tau+2,x) = \widehat{F}(\tau,x)
\]
for $\tau \in S$ are immediate consequences of
\[
K(\tau+2,z) = K(\tau,z) \qquad\text{and}\qquad \widehat{K}(\tau+2,z) = \widehat{K}(\tau,z).
\]
To prove the lemma, it will suffice to analytically continue  $\tau \mapsto
F(\tau,x)$ and $\tau \mapsto \widehat{F}(\tau,x)$ to some open neighborhood
of the closure of $S$ in $\Hyp$, such that the continuations satisfy
\[
F(\tau,x) + (-i \tau)^{-1/2} \widehat{F}(-1/\tau,x) = e^{\pi i \tau x^2}
\]
whenever $\tau$ and $-1/\tau$ are both in this neighborhood. Then we can use
the functional equations to extend these functions to all the hyperbolic
triangles in Figure~\ref{fig:poles}.\footnote{Note that as we pass from a
triangle to the adjacent triangles, we can never reach the same triangle via
two different paths of adjacencies, and thus we don't need to worry about
inadvertently defining a multivalued function of $\tau$.}

We can now use the information about poles and residues in
Proposition~\ref{prop:kernels}. When we analytically continue $F(\tau,x)$ to
$\tau$ just below the semicircle from~$-1$ to~$1$, the only relevant pole of
$z \mapsto K(\tau,z)$ is at $z=\tau$, since $-1/\tau$ is the only other
nearby point in the $\Gamma_\theta$-orbit of $\tau$, and there is no pole at
that point. We can set
\begin{equation} \label{eq:Fdeform}
F(\tau,x) = \int_{\mathcal{C}_\tau} K(\tau,z) e^{\pi i z x^2} \, dz,
\end{equation}
where $\mathcal{C}_\tau$ is a deformation of the semicircle to form a contour
from~$-1$ to~$1$ that passes below $\tau$, so that $\tau$ never lies on the
contour.

Similarly, we can analytically continue $\widehat{F}(\tau,x)$ to just below
the semicircle via
\[
\widehat{F}(\tau,x) = \int_{\mathcal{C}'_\tau} \widehat{K}(\tau,z) e^{\pi i z x^2} \, dz,
\]
where this time there is no pole at $x=\tau$, and the condition is that the
contour $\mathcal{C}'_\tau$ stays above the pole of $z \mapsto
\widehat{K}(\tau,z)$ at $z=-1/\tau$.

Now we can prove the functional equation
\[
F(\tau,x) + (-i \tau)^{-1/2} \widehat{F}(-1/\tau,x) = e^{\pi i \tau x^2}
\]
as follows
when $\tau$ is just below the semicircle. The identity
\[
K(-1/\tau,z) = -(-i \tau)^{1/2} \widehat{K}(\tau,z),
\]
or equivalently
\[
\widehat{K}(-1/\tau,z) = -(-i \tau)^{1/2} K(\tau,z),
\]
shows that
\begin{equation} \label{eq:Fhatdeform2}
\begin{split}
(-i \tau)^{-1/2} \widehat{F}(-1/\tau,x) &= \int_{\mathcal{C}'_{-1/\tau}} (-i \tau)^{-1/2} \widehat{K}(-1/\tau,z) e^{\pi i z x^2} \, dz\\
&= \int_{\mathcal{C}'_{-1/\tau}} - (- i \tau)^{-1/2} (-i \tau)^{1/2} K(\tau,z) e^{\pi i z x^2} \, dz\\
&= - \int_{\mathcal{C}'_{-1/\tau}} K(\tau,z) e^{\pi i z x^2} \, dz.
\end{split}
\end{equation}
Combining \eqref{eq:Fdeform} and \eqref{eq:Fhatdeform2} with the residue
theorem implies that
\[
F(\tau,x) + (-i \tau)^{-1/2} \widehat{F}(-1/\tau,x)
\]
is $2\pi i$ times the sum of the residues of all the poles of $z \mapsto
K(\tau,z) e^{\pi i z x^2}$ between $\mathcal{C}_\tau$ and
$\mathcal{C}'_{-1/\tau}$. The only pole that could lie between these contours
is at $z=\tau$, since $z \mapsto K(\tau,z)$ has no pole at $z=-1/\tau$, and
by construction it does lie between them. The residue of  $z \mapsto
K(\tau,z) e^{\pi i z x^2}$ at $z = \tau$ is $e^{\pi i \tau x^2}/(2\pi i)$,
and so
\[
F(\tau,x) + (-i \tau)^{-1/2} \widehat{F}(-1/\tau,x) = e^{\pi i \tau x^2},
\]
as desired.
\end{proof}

Lemma~\ref{lemma:analyticcont} shows that $F(\tau,x)$ and
$\widehat{F}(\tau,x)$ can be analytically continued to all $\tau \in \Hyp$ in
such a way that they satisfy the three functional equations. That is almost
everything we need to prove Theorem~\ref{theorem:interpolateR1} using
Lemma~\ref{lemma:reduction}. However, to apply this lemma we need to verify
certain growth conditions for $F(\tau,x)$ and $\widehat{F}(\tau,x)$ as $\tau$
approaches the real line. Verifying these conditions is the most technical
part of the proof of the interpolation theorem, and we will not examine it
here. In short, the verification combines bounds on $K$ and $\widehat{K}$
with careful accounting of how quickly the inhomogeneous terms from the third
functional equation can accumulate during the analytic continuation. Once
this is done, the proof of Theorem~\ref{theorem:interpolateR1} is complete.

This proof is satisfyingly thorough in that it not only proves the
interpolation formula but also provides plenty of additional information.
For example, we can obtain explicit formulas for the interpolation basis
$a_0$, $a_1$, \dots\  by using the identity $K(\tau+2,z) = K(\tau,z)$ to write $K$
as a Fourier series
\[
K(\tau,z) = \sum_{n=0}^\infty \varphi_n(z) e^{n \pi i \tau}
\]
when $\Im(\tau)$ is large.
Then
\[
a_n(x) = \int_{-1}^1 \varphi_n(z) e^{\pi i z x^2} \, dz,
\]
which generalizes Lemma~\ref{lemma:a0}. Similarly, the Fourier coefficients
of $\widehat{K}$ yield formulas for $\widehat{a}_n$.

On the other hand, some aspects of the proof are quite delicate. For example,
it is very sensitive to the form $\sqrt{n}$ of the interpolation points.
Specifically, the proof of the functional equation
\[
F(\tau,x) + (-i \tau)^{-1/2} \widehat{F}(-1/\tau,x) = e^{\pi i \tau x^2}
\]
depends on the fact that the complex Gaussian $x \mapsto e^{\pi i \tau x^2}$
equals $e^{n \pi i \tau}$ when evaluated at the interpolation point $x =
\sqrt{n}$. If we replaced $\sqrt{n}$ with other interpolation points $r_n$,
then the Fourier series for $F(\tau,x)$ would have to be replaced with
\[
\sum_{n=0}^\infty a_n(x) e^{r_n^2 \pi i \tau},
\]
and it would no longer satisfy $F(\tau+2,x) = F(\tau,x)$ if the values
$r_n^2$ are not integers. That would disrupt the algebraic mechanism behind
the proof.

Much remains to be understood regarding generalizations of the
Radchenko--Viazovska theorem and how Fourier interpolation fits into a broader
picture. One significant line of work \cites{BHMRV,BHMRV2} connects Fourier
interpolation to uniqueness theory for the Klein--Gordon equation
\cites{HM-R1,HM-R2,HM-R3}. Other noteworthy papers examine the density of
possible interpolation points \cites{K,S,KNS,A} and whether they can be perturbed
\cite{RS}, interpolation formulas using zeros of zeta and $L$-functions
\cite{BRS}, and extensions to nonradial functions \cites{St,RSt,RadSt,A}.
Perhaps the most surprising development so far has been a paper on sphere
packing and quantum gravity \cite{HMR}, which shows the equivalence of linear
programming bounds with the spinless modular bootstrap bound for free bosons
in conformal field theory, and which furthermore shows that certain bases of
special functions constructed by Maz\'a\v{c} and Paulos \cite{MP} for the
conformal bootstrap can be transformed into Fourier interpolation bases.

\section*{Acknowledgments}

I am grateful to Rupert Li and Danylo Radchenko for helpful feedback on a draft of this
article.

\section*{About the author}

Henry Cohn is a senior principal researcher at Microsoft Research New England and an
adjunct professor of mathematics at MIT. His primary research interests are in discrete
mathematics, with connections to computer science and physics.

\end{document}